\documentclass[12pt]{article}
\usepackage[textsize=tiny]{todonotes}

\usepackage{caption}
\usepackage{amsmath}                            
\usepackage{amssymb}                            
\usepackage{amsthm}                             
\usepackage{algorithm}
\usepackage[utf8]{inputenc}
\usepackage{enumitem}
\usepackage{centernot}
\usepackage{placeins}
\usepackage{amsthm}
\usepackage{amsmath}
\usepackage{mathrsfs}
\usepackage{amssymb}
\usepackage{scalefnt}
\usepackage{amsfonts}
\usepackage{xcolor}
\usepackage{color}
\usepackage{xspace}
\usepackage{microtype}
\usepackage[mathscr]{eucal}
\usepackage{float}
\usepackage{amsthm}

\DeclareMathOperator{\dcdot}{\cdot\cdot}
\usepackage{centernot}
\usepackage{tikz}
\usepackage{placeins}
\usepackage{amsthm}
\usepackage{amsmath}
\usepackage{amssymb}
\usepackage{scalefnt}
\usepackage{amsfonts}
\usepackage{color}
\usepackage{xspace}
\usepackage{microtype}
\usepackage[mathscr]{eucal}

\DeclareMathOperator{\Seq}{\subseteq}

\newtheorem{theorem}{Theorem}
\newtheorem{lemma} [theorem] {Lemma}
\newtheorem{proposition}[theorem]{Proposition}
\theoremstyle{theorem}

\newtheorem{corollary}[theorem] {Corollary}
\theoremstyle{definition}

\theoremstyle{theorem}

\usepackage{algorithm}
\usepackage{algpseudocode}
\def\BState{\State\hskip-\ALG@thistlm}
\makeatother
\newcounter{claimCount}
\newenvironment{claim}{\medskip

    \noindent\refstepcounter{claimCount}\textbf{Claim~\arabic{claimCount}.}}{

    \medskip}
\newenvironment{claimproof}{\noindent\textit{Proof of Claim~\arabic{claimCount}.}}{\hfill\ensuremath{\qedsymbol} \tiny{Claim~\arabic{claimCount}}

    \medskip}
\theoremstyle{definition}
\newtheorem{defn}{Definition}[section]

\theoremstyle{remark}
\newtheorem*{remark}{Remark}

\theoremstyle{definition}

\theoremstyle{theorem}

\theoremstyle{definition}
\newtheorem*{notation}{Notation}

\begin{document}
\title{Confining the Robber on Cographs}	





\author{
Masood Masjoody\\
\texttt{mmasjood@sfu.ca}
}

\maketitle

\begin{abstract} 

In this paper, the notions of {\em trapping} and {\em confining} the robber on a graph are introduced. We present some structural necessary conditions for graphs $G$ not containing the path on $k$ vertices (referred to as $P_k$-free graphs) for some $k\ge 4$, so that $k-3$ cops do not have a strategy to capture or confine the robber on $G$. Utilizing such conditions, we show that for planar cographs and planar $P_5$-free graphs the confining cop number is at most one and two, respectively. It is also shown that the number of vertices of a connected cograph on which one cop does not have a strategy to confine the robber has a tight lower-bound of eight. We also explore the effects of twin operations- which are well known to provide a characterization of cographs- on the number of cops required to capture or confine the robber on cographs. We conclude by posing two conjectures concerning the confining cop number of $P_5$-free graphs and the smallest planar graph of confining cop number of three.\\



\end{abstract}
\section{Introduction}\label{sec:intro}
For the definition of the game of cops and robbers and relevant basic definitions and terminology, see \cite{masjoody2020cops}. For other basic graph theoretic definitions see \cite{chartrand2019chromatic}. 

The game of cops and robbers on graphs with a forbidden induce subgraph was studied in \cite{CDM154}. The main results in \cite{CDM154} are summarized as follows:

\begin{theorem}\label{thm: joret}\cite{CDM154}\,
\begin{enumerate}
\item For a graph $H$, the class of $H$-free graphs is cop-bounded iff every component of $H$ is a path.
\item The class of $P_k$-free graphs ($k\ge 3$) is $(k-2)$-copwin.
\end{enumerate}
\end{theorem}

The results in \cite{CDM154} were extended \cite{masjoody2018CMS2,masjoody2020cops}, mainly through the introduction of the {\em Train-chasing Lemma} (Lemma \ref{lemma: train-chase-robber}), to the game of cops and robbers on graphs with a set of forbidden induced subgraphs.



\begin{defn}\cite{masjoody2020cops}
Let $G$ be a graph and $U$ be the set of all triples $(u,v,H)$ where $H$ is a connected subgraph of $G$, and $u,v\in V(H)$ with $d_H(u,v)\ge 2$. A {\em chasing function for} $G$ is a function $\theta$ mapping every triple $(u,v,H)\in U$ onto the neighbor of $u$ along a $(u,v)$-shortest path in $H$. 
\end{defn}

\begin{lemma}[\textbf{Train-chasing Lemma} \cite{masjoody2020cops}]\label{lemma: train-chase-robber}
Consider an instance of the game of cops and robber on a graph $G$. Let $\theta$ be a chasing function for $G$. Let $k\in \mathbb{N}$ and suppose on the cops' turn in step one there are $k$ cops $C_1,\dots,C_k$ in a vertex $v_1$ of the graph while the robber is located in a vertex $w_1$. Further, suppose the robber can and will play in such a way to survive the next $k$ steps of the game, regardless of how the cops $C_1,\dots, C_k$ play. Denote the following (generally not predetermined) robber's positions with $w_2,\dots,w_{k}$. Then, let $H_i$ ($i\in[1\dcdot k]$) and $v_i$ ($i\in[2\dcdot k]$) be defined recursively by the following relations:
\begin{itemize}
\item $H_1=G$;
\item $v_{i+1}=\theta(v_i,w_i,H_i)$ for $i\in[1\dcdot k]$;
\item $X_i=N_{H_i}(v_i)\setminus \{v_{i+1}\}$ for $i\in[1\dcdot k]$;
\item $H_{i+1}:$ the component of $v_1$ in $H_{i}-X_i$ for $i\in [1\dcdot k]$.
\end{itemize}
Then the following holds:
\begin{enumerate}
\item Every $H_i$ is an induced subgraph of $G$.
\item If $uv\in E(G)\setminus E(H_{k+1})$ such that $u\in V(H_{k+1})$, then $v\in\bigcup_{i=1}^{k}X_i$.
\item Vertices $v_1,\dots,v_{k+1}$, in that order, induce a path in $H_k$.
\item The cops can play such that on the cops' turn in step $k$  every $C_i$, $i\in[1\dcdot k]$, is located in vertex $v_i$.
\item Keeping every $C_i$ in $v_i$ for the rest of the game forces the robber to stay in $H_{k+1}$.
\end{enumerate}
\begin{center}
 \includegraphics[width=0.8\linewidth]{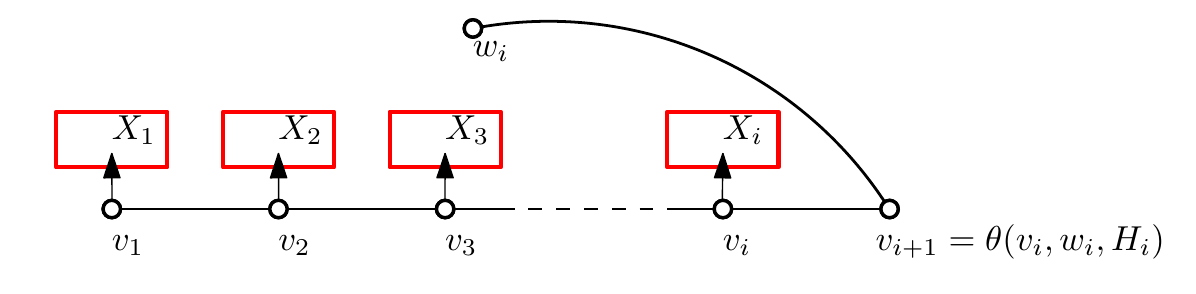}
 \captionof{figure}{Train-chasing the robber according to Lemma \ref{lemma: train-chase-robber} \cite{masjoody2020cops}}\label{pic: train-chase}
\end{center}
\end{lemma}

In \cite{masjoody2020cops}, the Train-chasing Lemma was in particular used to characterize classes $\mathscr{F}$ of graphs such that $\mathscr{F}$-free graphs are cop-bounded, under the condition that there is a constant bounding the diameter of the components of elements of $\mathscr{F}$. The resultant characterization generalizes Theorem \ref{thm: joret}(a). It is worth mentioning that the following extension of Theorem \ref{thm: joret}(b) is also an immediate corollary of the Train-chasing Lemma. (See \cite{yang2015optimal} for the definition of the {\em one-active-cop} version of the game of cops and robbers.)

\begin{theorem}\label{thm: Pk-free-Masood}\cite{masjoody2020cops}
For $k\ge 3$, $k-2$ cops require no more than $k-1$ steps of the game to capture the robber on a $P_k$-free graph in the one-active-cop version of the game of Cops and Robbers. 
\end{theorem}

In this paper we consider $P_k$-free graphs from the viewpoint of some new notions relevant to the cop number of graphs, described below.


\begin{defn}\label{def: conf.cop.no.} The {\em trapping cop number} of a graph $G$, denoted $tcn(G)$, is the minimum number of cops that can force an arrangement of the cops and the robber on vertices of $G$ in which the robber has to stay in the closed neighborhood $N_G[v]$ of a vertex $v$ in order to avoid capture in the very next move of the cops, in which case we say that the cops have {\em trapped} the robber. 
\end{defn}
\begin{defn}\label{def: conf.cop.no.} The {\em confining cop number} of a graph $G$, denoted $ccn(G)$, is the minimum number of cops that can force an arrangement of the cops and the robber on vertices of $G$ in which the robber has to stay in its position in order to avoid capture in the next move of the cops, in which case we say that the cops have {\em confined} the robber.  
\end{defn}

\begin{defn}\label{def: conf.cop.corner.} Let $G$ be a graph with $|G|\ge 3$. We call a vertex $v$ of $G$ a {\em confined corner} of $G$ if there exists a vertex $w$ such that $d_G(v,w)=2$ and $N_G(v)\Seq N_G(w)$, in which case $w$ is said to {\em confine} $v$ in $G$.  
\end{defn}

\begin{proposition}\label{prop: c(G)-ccn(G)}
For every graph $G$ one has $tcn(G)\le ccn(G)\le C(G)\le tcn(G)+1$.
\end{proposition}
\begin{proof}
The first two inequalities are obvious. As for the last one, note that with $ tcn(G)+1$ cops at hands, $ tcn(G)$ of them eventually force the robber to stay in $N_G[v]$ for some vertex $v$. By keeping those cops stationary and placing the remaining cop in $v$, the capture of the robber by the following step of the game will be guaranteed. 
\end{proof}

It is known that the cop number of any graph having girth $\ge 5$ is at least as large as its minimum degree:

\begin{proposition}[\cite{Aigner}]\label{prop: minimum_degree_cop_number}
For a graph $G$ with minimum degree $\delta$ one has $C(G) \ge \delta$ provided the girth of $G$ is at least 5.
\end{proposition}

The proof of Proposition \ref{prop: minimum_degree_cop_number} indeed establishes the following stronger result, which is in terms of the confining cop number of graphs.

\begin{proposition}\label{prop: minimum_degree_confining_cop_number}
For a graph $G$ with minimum degree $\delta$ one has $ccn(G) \ge \delta$ provided the girth of $G$ is at least 5.
\end{proposition}

\begin{corollary}
For every graph $G$ of order $\le 9$ one has $ccn(G)\le 2$. Moreover, the Petersen graph is the only graph on 10 vertices whose confining cop number is equal to 3.
\end{corollary}
\begin{proof}
As shown in \cite{baird2014minimum}, the cop number of the Petersen graph is three, whereas every graph $G$ on at most 10 vertices which is not the Petersen graph has $C(G)\le 2$. Moreover, by Proposition \ref{prop: minimum_degree_confining_cop_number}, the confining cop number of the Petersen graph is at least three. Hence, in light of Proposition \ref{prop: c(G)-ccn(G)} the desired claims follow.  
\end{proof}
In light of Proposition \ref{prop: c(G)-ccn(G)}, the following result can be presented as an extension of Theorem \ref{thm: Pk-free-Masood}. 
\begin{theorem}\label{thm: P-k-free-tcn}
If $G$ is a $P_k$-free graph for some $k\ge 3$, then $tcn(G)\le k-3$. Furthermore, $k-3$ cops need no more than $k-3$ steps of the game to trap the robber in the one-active-cop version of the game of cops and robbers.
\end{theorem}
\begin{proof}[Sketch of proof]
The proof is just an adaptation of the proof of Theorem \ref{thm: Pk-free-Masood} with $k-3$ cops in play. See \cite{masjoody2020cops} for details.
\end{proof}
\begin{remark}
The case $k=3$ is a triviality. Also, note that by Propositions \ref{prop: c(G)-ccn(G)} and Theorem \ref{thm: P-k-free-tcn}, for a $P_k$-free graph $G$ one has $tcn(G)>k-3$ iff $tcn(G)=C(G)=k-2$.
\end{remark}
\begin{notation}\label{notation: G-k-path-free}
Given $k\ge 4$, we will denote the class of all connected $P_k$-free graphs $G$ satisfying $ccn(G)=k-2$ (resp. $C(G)=k-2$) by $\mathscr{G}_{k,c}$ (resp. $\mathscr{G}_k$).
\end{notation}
\vspace{+5pt}
In Section \ref{sec: G-k} we will establish some necessary conditions for elements of $\mathscr{G}_k$ and $\mathscr{G}_{k,c}$. In light of such conditions, in Section \ref{sec: cographs} we will consider the game of cops and robbers on $P_4$-free graphs, also known as {\em cographs}. 


\begin{defn}\label{def: twin}
Distinct vertices $u,v$ in a graph $G$ are said to be {\em twins} (or to form a {\em twin pair}) if every other vertex in $G$ is adjacent to both $u$ and $v$, or non-adjacent to both $u$ and $v$. A pair $u,v$ of twin vertices in $G$ is called {\em true} (resp. {\em false}) whenever $N_G[u]=N_G[v]$ (resp. $N_G(u)=N_G(v)$).
\end{defn}


Several characterizations of cographs were established in \cite{corneil1981complement}, one of which states that a graph $G$ is a cograph iff every nontrivial induced subgraph of $G$ has a pair of twins. As one can easily see, the latter implies the following characterization, which is of our special interest in Section \ref{sec: cographs}:

\begin{theorem}\label{thm: cograph-twin}
A connected nontrivial graph $G$ is a cograph iff it can be obtained from $K_2$ by a sequence of twin operations.
\end{theorem}


\section{Some properties of $\mathscr{G}_k$ and $\mathscr{G}_{k,c}$}\label{sec: G-k}
One can easily see that $\mathscr{G}_{k,c}\Seq \mathscr{G}_k$. In that regard, first we establish some properties of $\mathscr{G}_k$. 

\begin{proposition}\label{prop: G-k C-k cycle}
Let $G\in \mathscr{G}_k$ and $v_1\in V(G)$. With $k-3$ cops at hands, suppose the robber uses any winning strategy against the cops. In addition, suppose the cops start at $v_1$ and play according to any chasing function $\theta$ for $G$ in the first $k-3$ steps of the game. Denote the position at the end of step $k-3$ of the robber by $w$. Let $H_i$ and $v_i$ be as in Lemma \ref{lemma: train-chase-robber}. Furthermore, for $j\in[1\dcdot (k-3)]$ let $$M_j:=N_{G}(v_j)\setminus \bigcup \{N_G[v_i]: 1\le i \le k-2, \;i\not=j \},$$
and for $j>k-3$ let $M_{j}$ be the $j$th neighborhood of $v_1$ in $H_{k-2}$. 
\begin{center}
 \includegraphics[width=0.8\linewidth]{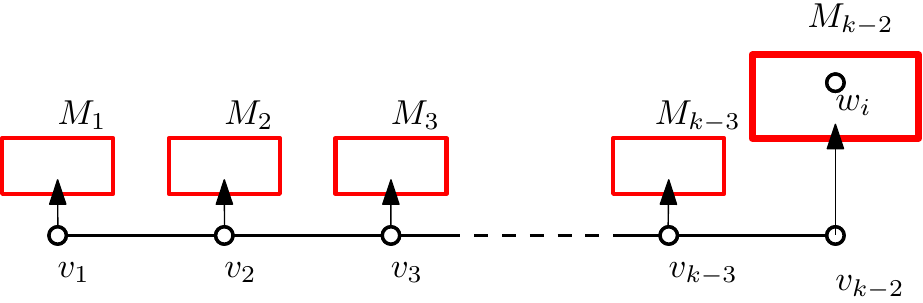}
 \captionof{figure}{An illustration of $M_j$s defined in Proposition \ref{prop: G-k C-k cycle}. }\label{pic: prop: G-k C-k cycle}
\end{center}
Then:
\begin{enumerate}
\item $M_j=\varnothing$ for $j\ge k-1$;
\item $M_j\not= \varnothing$ for each $j\in [1\dcdot(k-2)]$;
\item $M_1\Leftrightarrow M_{k-2}$; 
\item for each $u\in M_1$ and $z\in M_{k-2}$, $G[\{u,z,v_1,\dots,v_{k-2}\}]$ is a $k$-cycle; in particular, every vertex of $G$ belongs to an induced $k$-cycle; and
\item one has

\begin{equation}\label{eq: 1-1 prop: G-k C-k cycle}
w \in  \bigcap \{N_G(M_j): j\in[1\dcdot (k-3)]\}.
\end{equation}
In particular,
\begin{equation}\label{eq: 1 prop: G-k C-k cycle}M_{k-2}\cap \Big( \bigcap \{N_G(M_j): j\in[1\dcdot (k-3)]\}\Big)\not=\varnothing,\end{equation} 
and $G$ contains a vertex that belongs to an induced $C_j$ in $G$ for each $j\in[4\dcdot k]$.
\end{enumerate}

\end{proposition}
\begin{proof}
At the end of step $k-3$ of the game we have the cops along the induced path $P:\;v_1=v,v_2,\cdots,v_{k-3}$ in $H_{k-2}$, the robber at $w_0 \in M_{k-2} $-- hence, in particular, $M_{k-2}\not=\varnothing$-- and the game restricted to $H_{k-2}$ with the properties set forth in Lemma \ref{lemma: train-chase-robber}. In particular, if $M_j\not=\varnothing$ for some $j\ge (k-1)$, $H_{k-2}$ and, hence, $G$ would contain an induced $k$-path from $v_1$ to $M_j$, a contradiction. This establishes (a). Then, observe that since $v_{k-2}$ dominates $M_{k-2}$, as long as the cops cover the vertices of $P$ the robber has to stay in $M_{k-2}$. Moreover, if $M_j=\varnothing$ for some $j\in[1\dcdot (k-3)]$, then keeping cops in all $v_i$ with $i\in [1\dcdot (k-3)]\setminus \{j\}$ would still suffice to keep the robber in $M_{k-2}$, allowing the cops to cover all vertices in $\{v_i: i\in[1\dcdot (k-2)]\setminus \{j\}\}$ in the next step of the game; thereby capture the robber by the following step of the game. But this contradicts the assumption that $G\in\mathscr{G}_k$. Therefore, (b) also holds. Next, note that that if there exist $x\in M_1$ and $y\in M_{k-2}$ such that $xy\notin E(G)$, then $G[\{x,v_1,\dots,v_{k-2},y\}]$ would be a $k$-path, a contradiction. Hence, (c) must also hold. Note that (d) is immediate from (c) and the fact that any vertex $v\in V(G)$ can be set as the initial position $v_1$ of the cops. Finally, if given the position $w$ of the robber at the end of step $k-3$ of the game there exists $j_0\in[1\dcdot(k-3)]$ so that $w\not\in N_G(M_{j_0})$, then, as argued for (a), covering all vertices in $\{v_i: i\in [1\dcdot (k-2)]\setminus\{j\}\}$ by the cops forces the robber to stay within the neighborhood of at least one cop; thereby the robber will be captured by the very next step of the game; a contradiction. Hence, one has $$w\in \bigcap \{N_G(M_j): j\in[1\dcdot (k-3)\} ,$$ from which the other claims in (e) follow. 

\end{proof}

\begin{corollary}\label{corol: G-k 2-connected}
Every $G\in \mathscr{G}_k$ is 2-connected. 
\end{corollary}
\begin{proof}
In light of Proposition \ref{prop: G-k C-k cycle}(d), it suffices to show that no induced $k$-cycle in $G$ contains a cut-vertex of $G$. To this end, consider an induced $k$-cycle $C$ of $G$ and assume, toward a contradiction, that $C$ contains a cut-vertex $x$ of $G$. Let $B$ be the block of $G$ that contains $C$, and $B'$ be another block of $G$ that contains $x$. Pick a neighbor $y$ of $x$ in $C$, and any neighbor $z$ of $x$ in $B'$. Then, the graph
$$ G[(V(C)\setminus \{y\})\cup \{z\}]
$$
will be a $P_k$; a contradiction.
\end{proof}

\begin{proposition}\label{prop: G-k,c}
Let $G\in \mathscr{G}_{k,c}$ and $v_1\in V(G)$. We consider the assumptions and notations of Proposition \ref{prop: G-k C-k cycle} with the exception that we assume the robber uses any winning strategy against confinement by the cops.Then:
\begin{enumerate}
\item $|M_j|\ge 2$ for $j\in \{1,k-2\}$.
\item $E(G[M_j])$ is  nonempty  for $j\in\{1,k-2\}$.
\item $|V(G)|\ge 2k-2$.
\end{enumerate}
\end{proposition}
\begin{proof}
\textbf{(a) and (b)} Suppose the cops stay still after step $k-3$ of the game so that the robber has to stay in $M_{k-2}$ for the rest of the game. Since the robber's strategy avoids confinement, the robber at $w$ must have a neighbor $w'\in M_{k-2}$. Hence, $E(G[M_{k-2}])\not=\varnothing$ and $|M_{k-2}|\ge 2$. Likewise, by the cops occupying vertices $v_2,\dots,v_{k-2}$ in step $k-2$, the robber has to leave $w$ to a vertex $u\in M_1$ to avoid capture. Then, by the assumption $G\in \mathscr{G}_{k,c}$, keeping the cops stationary in the next step of the game leads to the existence of a vertex $u'$ satisfying $u'\in N_G(u)\setminus N_G(\{v_i:i\in[2\dcdot(k-2)\})$ to which the robber can move in step $k-1$ of the game. As such, considering the graph $$ G[\{u,u'\}\cup \{v_i:i\in[1\dcdot(k-2)\}]$$ shows that $u'$ must be in $N_G(v_1)$; thereby, $u'\in M_1$. As a result, we also have $E(G[M_{1}])\not=\varnothing$ and $|M_{1}|\ge 2$. (See Figure \ref{pic: 1 prop: G-k,c}.)
\begin{center}
 \includegraphics[width=0.8\linewidth]{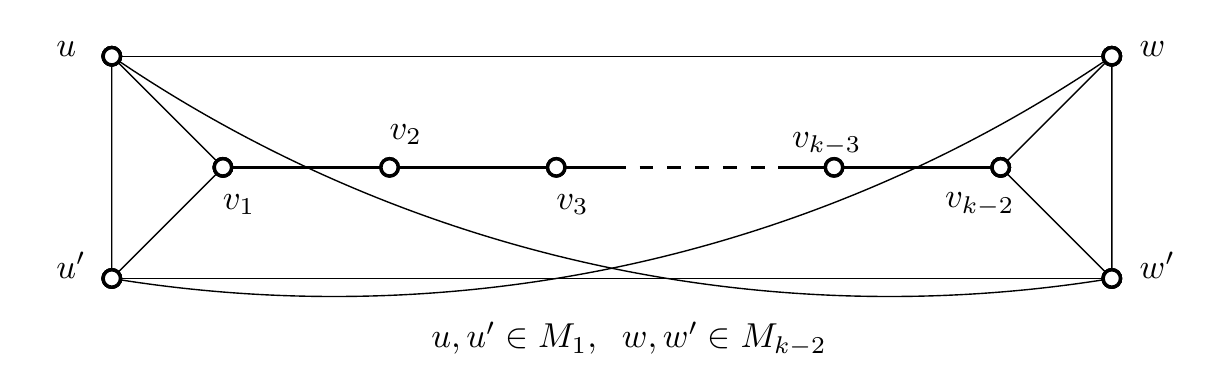}
 \captionof{figure}{An illustration for an induced subgraph of $G$ in Proposition \ref{prop: G-k,c}}\label{pic: 1 prop: G-k,c}
\end{center}
\textbf{(c)} Since the $k-1$ sets $M_1,\dots, M_{k-2}$ and $\{v_i: i\in[1\dcdot(k-2)\}$ are mutually disjoint subsets of $V(G)$, according to (a) and Proposition \ref{prop: G-k C-k cycle}\@(b) we have $$|V(G)|\ge  2(k-4)+2\times 3=2k-2 ,$$
as desired.
\end{proof}
\begin{corollary}\label{cor: -1 prop: G-k,c}
If a connected graph $G$ is planar and $P_k$-free for some $k\ge 4$ then $ccn(G)\le k-3$; in other words, every element of $\mathscr{G}_{k,c}$ is non-planar.
\end{corollary}
\begin{proof}
If $G\in \mathscr{G}_{k,c}$ then, in terms of the notations of Proposition \ref{prop: G-k,c} and its proof, one has $|M_1|,|M_{k-3}|\ge 2$ with $M_1\Leftrightarrow M_{k-3}$. Then, for any pair $\{u,u'\}$ and $\{w,w'\}$ of 2-subsets of $M_1$ and $M_{k-3}$ the graph $$G[\{u,u',w,w',v_1,\dots,v_{k-2}\}]$$ contains a subdivision of $K_{3,3}$ with partite sets $\{u,u',v_{k-2}\}$ and  $\{w,w',v_1\}$. (See Figure \ref{pic: 1 prop: G-k,c}.) Hence, $G$ is non-planar according to the Kuratowski Theorem.
\end{proof}
\noindent The following is also immediate in light of Proposition \ref{prop: G-k,c}.
\begin{corollary}\label{cor: prop: G-k,c}
For every $G\in\mathscr{G}_{k,c}$ one has $\delta(G)\ge 3$ and $\Delta(G)\ge k$. 
\end{corollary}
\begin{proof}
We implement the notations of Proposition \ref{prop: G-k,c} and its proof. In that regard, for any typical vertex $v_1$ of $G$ one has $N_G(v_1)\supseteq \{u,u',v_2\}$. Hence, $\delta(G)\ge 3$. Furthermore, since $\{w',u,u'\}\Seq N_G(w)$ and $N_G(w)\cap M_j$ is non-empty for each $j\in[1\dcdot (k-3)]$, one also has $|N_G(w)|\ge 3+ (k-3)=k$. Thus, $\Delta(G)\ge k$, as desired.
\end{proof}%
\noindent For $k\ge 5$ one can strengthen the first part of Corollary \ref{cor: prop: G-k,c}:
\begin{proposition}\label{prop: G-k,c;k.g.e.5}
For every $G\in\mathscr{G}_{k,c}$ with $k\ge 5$ one has $\delta(G)\ge 4$.
\end{proposition}
\begin{proof}
Toward a contradiction, let $G\in \mathscr{G}_{k,c}$ with $\delta(G)\le 3$. Then, by Corollary \ref{cor: prop: G-k,c}, $\delta(G)=3$. Pick any vertex $v_1\in V(G)$ with $\deg_G(v_1)=3$. Let there be $k-3$ cops at hands. Then, with the assumptions and notations of Propositions \ref{prop: G-k C-k cycle} and \ref{prop: G-k,c}, one has $N_G(v_1)=\{u,u',v_2\}$ and $M_1=\{u,u'\}$. (See Figure \ref{pic: 1 prop: G-k,c}.) Having cops at vertices $v_2,\ldots,v_{k-2}$ in step $k-2$ of the game forces the robber to move to one of the vertices in $M_1$, say $u$. Then, in the following step, moving the cop at $v_{k-2}$ to $w$ and keeping the other cops stationary forces the robber to move to a neighbor, say, $z$ of $u$ so that the robber will avoid being captured in the very next cop moves. Then, one must have $z\in V(G)\setminus N(v_j)$ for each $j\in[2\dcdot (k-3)]$. Moreover, by Proposition \ref{prop: G-k C-k cycle}(c), $z\not\in M_1$; thereby $z\not\in N(v_1)$. Therefore, $z$ must be a non-neighbor of $w$ in $M_{k-2}$, for otherwise $G[\{v_j:j\in[1\dcdot(k-2)]\}\cup \{u,z\}]$ would be a $P_k$. 
\begin{center}
 \includegraphics[width=0.8\linewidth]{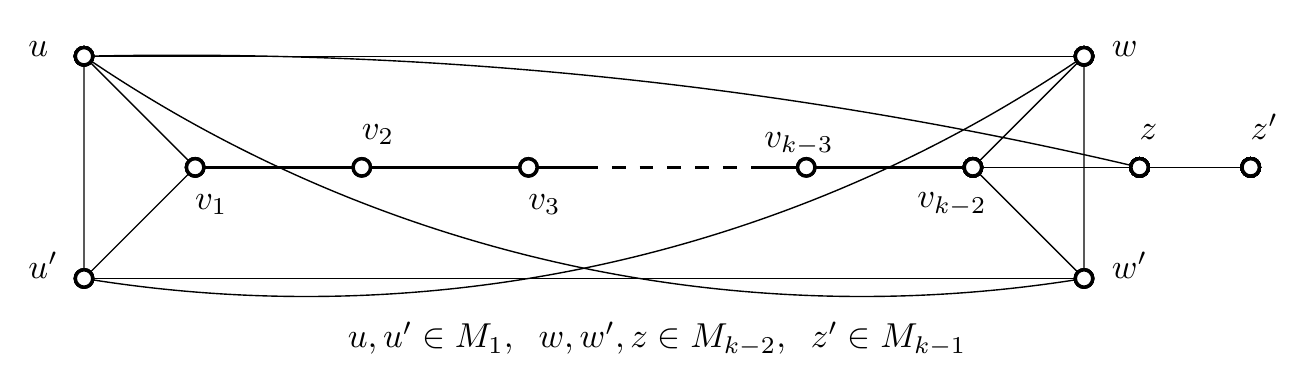}
 \captionof{figure}{Proof of Proposition \ref{prop: G-k,c;k.g.e.5} by contradiction}\label{pic: 1 prop: G-k,c;k.g.e.5}
\end{center}
Then, in the next step of the game, move the cop at $w$ to $u$ while keeping the rest of the cops stationary. This forces the robber to move to a neighbor $z'$ of $z$ in order to avoid capture. As with $z$, now one must have $z'\in V(G)\setminus N(v_j)$ for each $j\in[2\dcdot (k-3)]$, and $z'\not\in M_{j}$ for $j\in\{1,k-2\}$. Consequently, one also has $z'\in V(G)\setminus N(v_j)$ for $j\in[1\dcdot(k-2)]$; i.e. $z'\in M_{k-1}$. The latter contradicts Proposition \ref{prop: G-k C-k cycle}(a). (See Figure \ref{pic: 1 prop: G-k,c;k.g.e.5} for an illustration.)
\end{proof}
\section{Cops and Robbers on Cographs}\label{sec: cographs}
\noindent By Theorem \ref{thm: P-k-free-tcn}, for every cograph $G$ one has $tcn(G)=1$. In this section, we consider the effects of twin operations on the cop number and confining cop number of cographs.
\begin{proposition}\label{prop: twin-cn-cograph}
Let $G_1$ be a cograph and $x\in V(G_1)$.
\begin{enumerate}
\item If $G_2$ is obtained from $G_1$ by adding a true twin $y$ of $x$, then $c(G_1)=c(G_2)$.
\item If $G_3$ is obtained from $G_1$ by adding a false twin $z$ of $x$, then $c(G_1)\le c(G_3)$.
\end{enumerate}
\end{proposition}

\begin{proof}
We will use the fact that a graph is copwin iff it is dismantlable. \textbf{(a)} Since $N_{G_2}(x)=N_{G_2}(y)$, $G_1$ is a one-point retract of $G_2$ and, hence, $C(G_1)\le C(G_2)$. As a result, if $c(G_1)=1$ then pasting $x$ in front of any elimination ordering of $G_1$ gives an elimination ordering of $G_2$; therefore $c(G_2)=1$. Moreover, one also has $c(G_1)=c(G_2)$ whenever $c(G_1)=2$, since $C(G_1)\le C(G_2)$ and cographs are cop-bounded by two. \textbf{(b)} By the fact that cographs are cop-bounded by two, one only needs to consider the case where $c(G_1)=2$. In this case, the robber has a winning strategy $\mathscr{S}$ against one cop on $G_1$. Then on $G_3$ and against one cop, the robber can react to any move of the cop to or from $y$ as if the cop has moved to or from $x$ and, as such, simply move among $V(G_1)$ according to $\mathscr{S}$. One can easily check that the latter is a winning strategy for the robber on $G_2$; therefore, $C(G_3)=2$ whenever $c(G_1)=2$.
\end{proof}
\begin{remark}
Note that the false twin operation can indeed increase the cop number of a cograph, as is the case with $C_4$ (with $c(C_4)=2$) which is obtained by the false twin operation on the degree-two vertex of the copwin graph $P_3$.
\end{remark}

\begin{theorem}\label{thm: twin-cn-cograph}
Let $G_1$ be a cograph and $x\in V(G_1)$.
\begin{enumerate}
\item If $G_2$ is obtained from $G_1$ by adding a true twin $y$ of $x$, then one has $ccn(G_1)\le ccn(G_2)$.
\item If $G_3$ is obtained from $G_1$ by adding a false twin $z$ of $x$, then one has $ccn(G_1)= ccn(G_3)$.
\end{enumerate}
\end{theorem}
\begin{proof}
\textbf{(a) }It suffices to consider the case where $ccn(G_1)=2$ so that the robber has a strategy $\mathscr{S}$ against one cop on $G_1$ to avoid confinement. Then, the robber can mimic $\mathscr{S}$ on $G_2$, as shown in the proof of Proposition \ref{prop: twin-cn-cograph}\@(b), to avoid confinement by one cop on $G_2$. Therefore, $ccn(G_2)=2$ when $ccn(G_1)=2$. \textbf{(b) }Likewise the proof of (a), one can easily see that $ccn(G_3)=2$ whenever $ccn(G_1)=2$. Hence, in any case we have $ccn(G_1)\le ccn(G_3)$. Therefore, to complete the proof,  we assume $ccn(G_1)=1$ and $ccn(G_3)=2$, and show that these assumptions together give rise to a contradiction. To this end, consider a fixed strategy $\mathscr{S}'$ for one cop leading to confining or capturing the robber on $G_1$. Then, in the game of cops and robbers on $G_3$ with one cop, move the cop within $V(G_1)$ by using the following strategy shadowing $\mathscr{S}'$: If the robber moves to or from $z$, follow $\mathscr{S}'$ pretending that the robber has moved to or from $x$. Eventually, the game will reach a situation corresponding to confining or capturing the robber on $G_1$. The latter case, in turn, corresponds to the capture of the robber on $G_3$ unless the cop and the robber on $G_3$ are located at $x$ and $z$, respectively, in which case the robber is confined by the cop. Hence, we may assume the game on $G_3$ has reached a situation corresponding to the confinement of the robber on $G_1$ with the robber and the cop positioned at vertices, say, $x'$ and $y'$ with $d_{G_1}(x',y')=2$. Pick a vertex $z'\in N_{G_1}(x')\cap N_{G_1}(y')$. Keep in mind that $z\not\in \{x',y',z'\}$. If the position of the robber in the actual game (i.e. the game on $G_3$) is not $x'$, then it has to be $z$, in which case $x'=x$ and $ N_{G_3}(y')=N_{G_1}(y')$. Then, in light of the latter one gets 
\begin{equation}\label{eq 1: thm: twin-cn-cograph} 
N_{G_3}(y')\supseteq  N_{G_1}(x')=N_{G_3}(x')=N_{G_3}(z);
\end{equation}
consequently, in the game on $G_3$ the cop (at $y'$) has also confined the robber (at $z$). (See Figure \ref{eq 1: thm: twin-cn-cograph} for an illustration.) 
\begin{center}
 \includegraphics[width=0.6\linewidth]{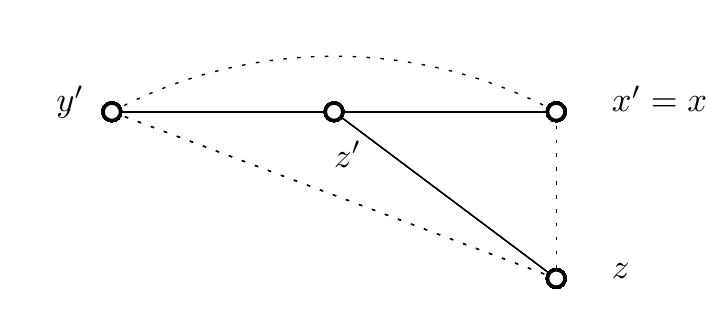}
 \captionof{figure}{An illustration of the situation leading to \eqref{eq 1: thm: twin-cn-cograph}} \label{pic: 1 prop: cograph- ccn=2}
\end{center}
Therefore, it only remains to deal with the case where the robber in the actual game is also positioned is at $x'$. If, in addition, one has $z\not\in N_{G_3}(x')$ or $z\in N_{G_3}(x')\cap N_{G_3}(y')$, then $N_{G_3}(x')\supseteq N_{G_3}(y')$, implying that the robber has been confined on $G_3$. Hence we may assume

\begin{equation}\label{eq: 1 thm: cographs-CR}
z\in N_{G_3}(x')\setminus N_{G_3}(y').
\end{equation} 
To complete the proof, we will show that this assumption leads to a contradiction. (See Figure \ref{pic: 2 prop: cograph- ccn=2} for an illustration of the following argument.) To this end, first note that considering $G_3[\{x',y',z',z\}]$ gives
\begin{equation}\label{eq: 2 thm: cographs-CR}
zz'\in E(G_3),
\end{equation}
since $G_3$ is $P_4$-free. As a result, we also have
\begin{equation}\label{eq: 2+1 thm: cographs-CR}
xx',xz'\in E(G_3)
\end{equation}
since $x$ and $z$ are twins in $G_3$. Note that by \eqref{eq: 1 thm: cographs-CR} and \eqref{eq: 2 thm: cographs-CR} we have $x'\not=x$ and $z'\not=x$. Furthermore, since $x'\in N_{G_3}(z)\setminus N_{G_3}(y')$ and vertices $z$ and $x$ are false twins in $G_3$, we also have $y'\not= x$. Hence, we have 
\begin{equation*}\label{eq: 3 thm: cographs-CR}
x\not\in \{x',y',z'\}.
\end{equation*}
Moreover, since $N_{G_1}(y')\supset N_{G_1}(x')$ we have
\begin{equation}\label{eq: 5 thm: cographs-CR}
xy'\in E(G_3).
\end{equation} 
\begin{center}
 \includegraphics[width=0.6\linewidth]{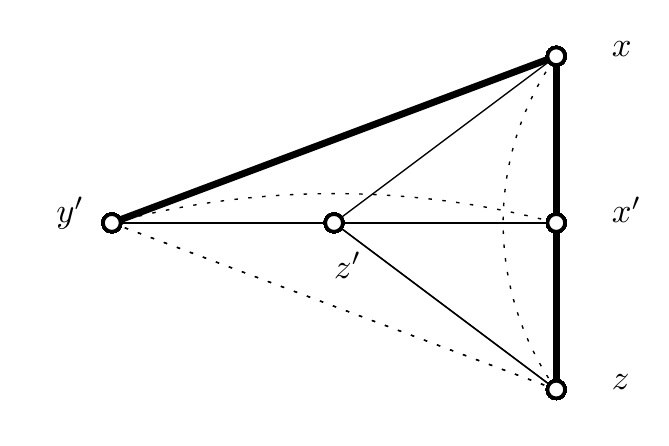}
 \captionof{figure}{An illustration of the contradiction arisen from assuming \eqref{eq: 1 thm: cographs-CR}} \label{pic: 2 prop: cograph- ccn=2}
\end{center}
Finally, in light of \eqref{eq: 1 thm: cographs-CR}, \eqref{eq: 2+1 thm: cographs-CR}, and \eqref{eq: 5 thm: cographs-CR} one gets $G_3[\{z,x',y,y'\}]\cong P_4$, a contradiction.
\end{proof}

\begin{corollary}\label{corr: thm: twin-cn-cograph} 
If $G$ is a cograph with $ccn(G)=2$, then for every graph $H$ obtained from $G$ by a sequence of twin operations one has $ccn(H)=2$.
\end{corollary}
Adding a true twin vertex to a cograph can indeed increase the confining cop number. This claim, according to Theorem \ref{thm: cograph-twin} and Theorem \ref{thm: twin-cn-cograph}, is equivalent to the statement that there exists a cograph $G$ with $ccn(G)=2$. We shall show that the smallest order of such a graph is eight:

\begin{theorem}\label{thm: cograph- ccn=2}
The confining cop number of every connected cograph on fewer than 8 vertices is equal to one. Moreover, for every $n\ge 8$ there is a connected cograph $G$ on $n$ vertices such that $ccn(G)=2$. 
\end{theorem}
\begin{proof}
Let $G$ be a graph in $\mathscr{G}_{4,c}$ with the minimum number of vertices. By Proposition \ref{prop: G-k,c}\@(c), one has $|V(G)|\ge 6$. Indeed, by Proposition \ref{prop: G-k,c} and in accordance with its notations, $G$ must have the graph $G_1$ of Figure \ref{pic: 1 thm: cograph- ccn=2} as an induced subgraph. 

\begin{center}
 \includegraphics[width=0.65\linewidth]{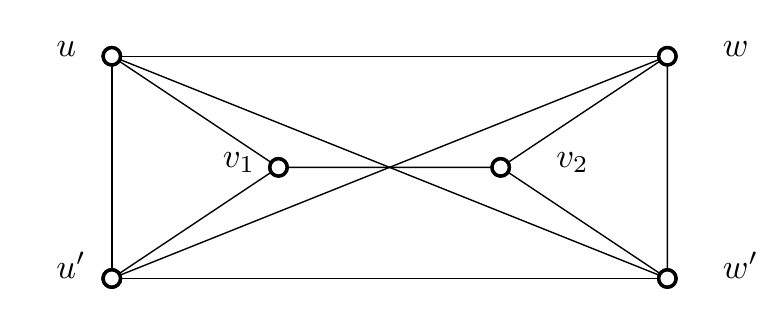}
 \captionof{figure}{The induced subgraph $G_1$ of $G$ in the Proof of Theorem \ref{thm: cograph- ccn=2}} \label{pic: 1 thm: cograph- ccn=2}
\end{center}
As such, if $|V(G)|=6$, one has $G=G_1$, in which case placing a cop at $w$ in the first step of the game forces the robber to chose $v_1$ as its first position, at which vertex the robber is confined; a contradiction. Hence, we have $$ |V(G)|\ge 7.$$ Next, we will show that $| V(G)|\ge 8 $. To this end, we show that each of the following three cases gives rise to a contradiction: 
\begin{description}
\item[Case 1:] $|V(G)|= 7$ and $|M_1|=3$.
\item[Case 2:] $|V(G)|= 7$, $|M_1|=2$, and $\deg_G(v_1)=4$.
\item[Case 3:] $|V(G)|= 7$ and $|M_2|=3$.
\end{description}
\underline{Case I:} In this case one can easily examine that placing a cop at $w_1$ leads to the confinement or capture of the robber, hence, $ccn(G)=1$; a contradiction.\\[3pt]
\noindent\underline{Case II:} Let $\{x\}=N_G(v_1)\setminus (M_1\cup \{v_2\})$. Since $x\not\in M_1$, we have $x\in N_G(v_2)$. If $x$ is adjacent to a vertex in $M_1$ (resp. $M_2$), placing a cop at that vertex leads to either the confinement of the robber at $v_2$ (resp. $v_1$) in step 1 or the capture of the robber in step 2; a contradiction. Hence, $N_G(x)=\{v_1,v_2\}$. But then the the graphs $G[\{u,w,v_2,x\}]$ will be a $P_4$; contradicting the assumption that $G$ is $P_4$-free. (See Figure \ref{pic: 2 thm: cograph- ccn=2}.)
\begin{center}
 \includegraphics[width=0.56\linewidth]{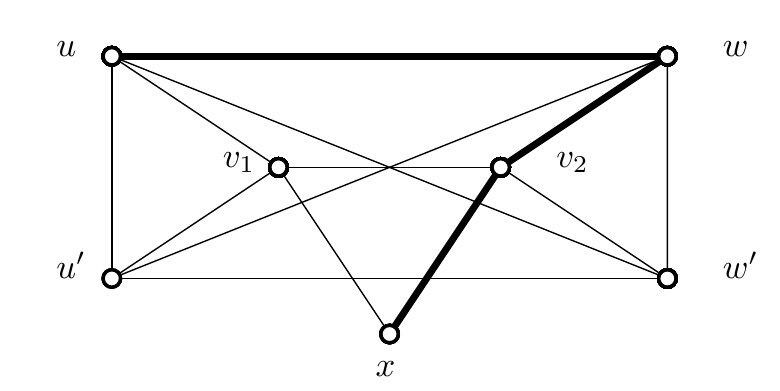}
 \captionof{figure}{Proof of Theorem \ref{thm: cograph- ccn=2}: An induced $P_4$ under Case II} \label{pic: 2 thm: cograph- ccn=2}
\end{center}
\underline{Case III:} This case also leads to a contradiction; likewise Case I.\\[3pt]
Hence, $|V(G)\ge 8$. Therefore, in light of Corollary \ref{corr: thm: twin-cn-cograph} , to complete the proof it suffices to present a cograph of $H$ of order eight so  that $ccn(H)=2$. As one can easily check, the graph of Figure \ref{pic: 3 thm: cograph- ccn=2} satisfies these conditions; indeed, it is the only cograph of order eight with the confining cop number of two. 
\begin{center}
 \includegraphics[width=0.55\linewidth]{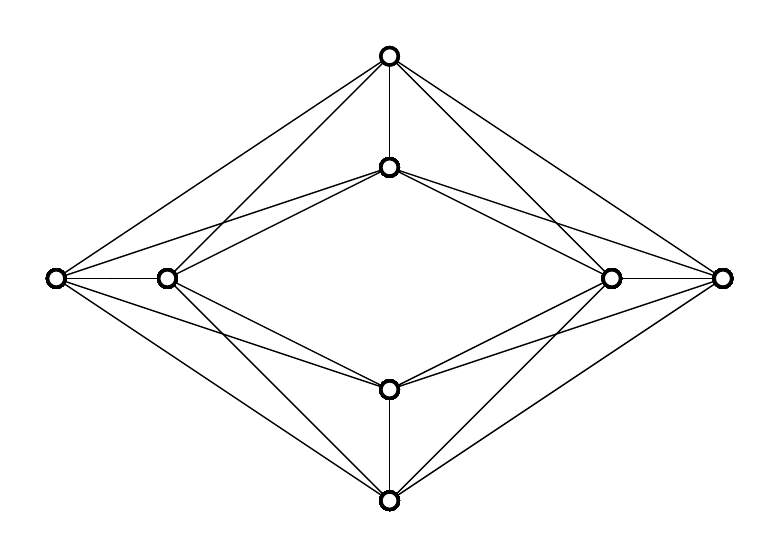}
 \captionof{figure}{The smallest cograph with the confining cop number of two} \label{pic: 3 thm: cograph- ccn=2}
\end{center}
\end{proof}

\section{Concluding remarks}

Since the cop number of the cycle on four vertices is two, the upper bound of two for the cop number of cographs is tight. However, it is an open question whether there exists a $P_5$-free graph which requires three cops to capture the robber. As shown in Corollary \ref{cor: -1 prop: G-k,c}, though, the planarity of a connected $P_5$-free graph $G$ implies $ccn(G)\le 2$. We conjecture that this planarity condition can be relaxed:\\[3pt]

\noindent\textbf{Conjecture.} {\em For every connected $P_5$-free graph $G$, one has $ccn(G)\le 2$.}\\

We conclude by another conjecture about the planar graphs. In light of Propositions \ref{prop: c(G)-ccn(G)} and \ref{prop: minimum_degree_confining_cop_number} and the fact that planar graphs are 3-copwin, one can easily see that the dedecahedral graph has its cop-number and confining cop-number both equal to three. It has been conjectured that the dodecahedral graph (which has 20 vertices) is the smallest planar graph with cop-number three. Here, we pose the counterpart of this conjecture in terms of the confining cop number:\\

\noindent\textbf{Conjecture.} {\em For every connected planar graph $G$ on at most 19 vertices one has $ccn(G)\le 2$.





\end{document}